\documentclass[12pt,reqno]{amsart}
\usepackage[english]{babel} \usepackage{latexsym, amsfonts, amsmath, amssymb, amsthm, mathrsfs, color}
\usepackage{lmodern} \usepackage{mathpazo} \usepackage[pdftex]{hyperref}
\usepackage[top=2.5cm, bottom=2.5cm, left=2.5cm, right=2.5cm]{geometry}
\numberwithin{equation}{section}  \makeatletter\@addtoreset{equation}{section}
\newtheorem {theorem}{Theorem}[section]         \newtheorem {lemma}[theorem]{Lemma}     
\newtheorem {corollary}[theorem]{Corollary}     \newtheorem {remark}[theorem]{Remark}   
       \newtheorem {proposition}[theorem]{Proposition}
\newcommand{\C}{\mathbb C}    \newcommand{\R}{\mathbb R}   \newcommand{\Z}{\mathbb Z} 	
  
\newcommand{\Fnug}{\mathcal{F}^{2}_{\Gamma, \chi}}
\newcommand{\Fnugr}{\mathcal{F}^{2}_{\Gamma_r, \chi}}
\newcommand{\Fnugd}{\mathcal{F}^{2}_{\Gamma_{2g}, \chi}}

\newcommand{\VgR}{\mathbb{V}^{2g}_\R}
\newcommand{\VgC}{\mathbb{V}^{g}_\C}
\newcommand{\VuC}{\mathbb{V}^{1}_\C}
\newcommand{\VuoC}{{\mathbb{V}_\C^1}^{\perp}}

\newcommand{\normnu}[1]{\left\Vert#1\right\Vert_{\Gamma}}
\newcommand{\normnur}[1]{\left\Vert#1\right\Vert_{\Gamma_r}}
\newcommand{\scalnur}[1]{\left<#1\right>_{\Gamma_r}}
\newcommand{\scalnue}[1]{\left<#1\right>_{\Gamma}}
\newcommand{\mes}{d\lambda}
\newcommand{\nuw}{\nu}
\newcommand{\epf}{e^{\alpha,\nuw}_{n}}
\newcommand{\epfo}{e^{\alpha,\nuw}}
\newcommand{\epfon}{e^{\alpha+n,\nuw}}
\newcommand{\eqf}{e^{\alpha,\nuw}_{n'}}

\begin{document}

\title[]{Concrete description for rank one $(\Gamma, \chi)$-theta Fock-Bargmann space in high dimension}
\author{M. Souid El Ainin}
 \address{Department of Mathematics,  P.O. Box 1014,   Faculty of Sciences,  Mohammed V-Agdal University,  Rabat,  Morocco}
\email{msouidelainin@yahoo.fr}

\begin{abstract}
We deal with the $(\Z\omega, \chi)$-theta Fock-Bargmann space consisting of holomorphic
automorphic functions associated to given discrete subgroup in $\C^g$ of rank one and given character $\chi$.
We give concrete description of its elements, it looks like a tensor product of a theta Fock-Bargmann space on the complex plane $\C$
and the classical Fock-Bargmann space on $\C^{g-1}$. Moreover, we construct an orthonormal basis and we give explicit expression of its
reproducing kernel function in terms of Riemann theta function.
\end{abstract}

\maketitle

\section{Introduction}

 Let $\VgC=\C^{g}$ be a $g$-dimensional complex vector space endowed with a positive definite hermetian form $H(u, v )$.
Associated to  discrete subgroup $\Gamma_r$  of rank $r$; $0\leq r \leq 2g$,  in the $2g$-real vector space $\VgR=\R^{2g}$,
we consider the space of holomorphic $(\Gamma_r,\chi)$-automorphic functions consisting of the complex valued holomorphic functions on $\VgC$, $f\in \mathcal{O}(\VgC)$,  displaying the functional equation
\begin{equation}\label{IntrFctEq1}
f(u+\gamma) = \chi(\gamma)  e^{ H\left(u+\frac{\gamma}2, \gamma\right) } f(u);\quad u\in \VgC,  \gamma\in\Gamma_r,
\end{equation}
where $\chi$ is a given mapping on $\Gamma_r$ such that $|\chi(\gamma)|=1$.
We define the norm  $\normnur{f}$
to be the one associated to the inner scaler product
 \begin{equation}\label{IntrScal1}
\scalnur{f,g} := \int_{\Lambda(\Gamma_r)} f(u) \overline{g(u)}  e^{- H(u, u)} \mes(u),
\end{equation}
where $\mes(u)$ is the usual Lebesgue measure on $\VgC$ and $\Lambda(\Gamma_r)$ is a fundamental domain of $\Gamma_r$, which represents in $\VgR$ the orbital group $\VgR/\Gamma_r$ with respect to its quotient topology. In fact, for every $f,g$ satisfying \eqref{IntrFctEq1},
the function $ f(u) \overline{g(u)} e^{- H(u, u)}$ is $\Gamma_r$-periodic the quantity in \eqref{IntrScal1}
makes sense and is independent of the choice of $\Lambda(\Gamma_r)$.
Then one performs the functional space $\Fnugr(\VgC)$ of $(L^2, \Gamma_r, \chi)$-theta functions, i.e, the space of all holomorphic
$(\Gamma_{r},\chi)$-automorphic functions on $\VgC$ satisfying the growth condition $\normnur{f}^2 < +\infty$. That is
 \begin{align}\label{Space}
\Fnugr(\VgC) := \left\{f\in \mathcal{O}(\VgC) \mbox{ displaying }  \eqref{IntrFctEq1} \mbox{ and }
\normnur{f}^2 := \scalnur{f,f} < +\infty
             \right\}.
\end{align}
Note for instance that for $\Gamma_0=\{0\}$,  equation \eqref{IntrFctEq1} reads simply $ f(u) = \chi(0) f(u)$ so that $\chi=1$.
Furthermore, since $\Lambda(\Gamma_0)=\VgC$, the corresponding $\mathcal{F}^{2}_{\Gamma_0, \chi}(\VgC)$
reduces further to the classical Bargmann space on $(\VgC, H)$, i.e;
\begin{equation}\label{Bargmann}
 \mathcal{B}^{2}_H(\VgC) := \left\{ f\in \mathcal{O}(\VgC) ; \quad \int_{\VgC} |f(u)|^2 e^{- H(u, u)} \mes(u)<+\infty \right\}.
 \end{equation}
 For maximal rank ($r=2g$), the $\Lambda(\Gamma_{2g})$ is compact and the space $\Fnugd(\VgC)$ is nontrivial under the Riemann-Dirac quantization condition
 $\chi(\gamma+\gamma')=\chi(\gamma)\chi(\gamma')e^{i \Im( H(\gamma, \gamma'))}$; $\gamma, \gamma'\in \Gamma_{2g},$
 and is of finite dimension that is given by explicit formula involving the volume of the complex torus $\C^g/\Gamma_{2g}$.
The spectral properties relevant to the $(L^2, \Gamma_{2g},  \chi)$-theta functions are investigated in \cite{Bump1996,GhIn2008}.
This was worked out by rather different methods in \cite{Folland2004} for $\chi\equiv 1$.

In spite of the fact that some analytical aspects of the $(L^2,\Gamma_{r},\chi)$-automorphic functions for lower rank ($r=0$) and maximal
rank ($r=2g$) were known and were well discussed in the literature, the general problem for nontrivial non-compact lattices (i.e. $\Gamma_{r}$ is of rank $1\leq r \leq 2g $) is still far from being understood.
The first investigation in this direction has being discussed recently in \cite{GhIn13JMP} dealing
with the rank $1$ in two dimension. It is shown there that for $\chi$ being a character the  $(\Z,\chi)$-theta Fock-Bargmann space $\mathcal{F}^{2,\nu }_{\Gamma, \chi}(\C)$ is of infinite dimension and an orthonormal basis of such Hilbert space is constructed explicitly, to
 wit
 $$\psi_n^{\alpha,\nu} (z) :=  \left(\frac{2\nu}{\pi}\right)^{1/4} e^{\frac{\nu}2 z^2  + 2i\pi (\alpha + n) z -\frac{\pi^2}{\nu}(n+\alpha)^2 },$$
 for varying $n \in\Z$ and some fixed real number $\alpha$.
Moreover explicit expression of the reproducing kernel is given there in terms of the modified theta function
\begin{equation}\label{RiemannTheta}
 \theta_{\alpha, \beta}(z|\tau)    : =  \sum_{n\in \Z} e^{i\pi (n + \alpha)^2 \tau + 2i\pi (n + \alpha)(z + \beta)},
\end{equation}
where $z\in \C$ and $\tau$ is confined to the upper half-plane.
In \cite{InAb2013}, chaoticity of a shift operator on the obtained basis is studied by Intissar.

In the present paper we follow in spirit \cite{GhIn13JMP} by Ghanmi and Intissar in order to generalize their results to high dimension. More precisely, we deal with the space $\mathcal{F}^{2}_{\Gamma, \chi}(\VgC)$ corresponding to the nontrivial discrete subgroup $\Gamma=\Z \omega $ in $\VgC$ with  $g\geq 2$. Our main aim is to construct a concrete orthonormal basis of $\mathcal{F}^{2}_{\Gamma, \chi}(\VgC)$ and to explicit its reproducing kernel. In fact, the space $\mathcal{F}^{2}_{\Gamma, \chi}(\VgC)$ looks like a tensor product of $\mathcal{F}^{2}_{\Z \omega, \chi}(\C \omega)$ defined through \eqref{Space}
with the classical Fock-Bargmann space  $\mathcal{B}^{2}_H(\VuoC)$ on $\VuoC$, as defined in \eqref{Bargmann}. Namely, Theorem \ref{ThmExpansion} characterizes the $(\Z\omega, \chi)$-theta Fock-Bargmann space $\Fnug(\VgC)$ as the space of all series
\begin{equation*}
      \sum\limits_{n\in\Z} \psi_n(z') e^{\frac{{{\nuw}} }2 z^2 + 2i\pi (\alpha + n) z},
     \end{equation*}
satisfying the growth condition
      \begin{align*}
       \left(\dfrac{\pi}{2{{\nuw}} }\right)^{1/2}\sum\limits_{n\in\Z}e^{\frac{2\pi^2}{\nuw}(n+\alpha)^2}
       \|\psi_n\|_{L^2(\VuoC,  \widetilde{H})}^{2}  <+\infty ,
       \end{align*}
 for some holomorphic functions  $\psi_n$ and where $\|\psi_n\|_{L^2(\VuoC,\widetilde{H})}^{2} $ denotes
 the square norm of $\psi_n$ in the Hilbert space $L^2(\VuoC, e^{- \widetilde{H}(\widetilde{u},\widetilde{u})}\mes(\widetilde{u}))$.
Moreover,  Theorem \ref{ThmbasisO} shows that  $\Fnug(\VgC)$ is a reproducing kernel Hilbert space and the set of functions
 \begin{equation}
 e^{\alpha,{\nuw} }_{n,\mathbf{k}}(z,z')=:e^{\frac{\nuw}2 z^2 + 2i\pi (\alpha + n) z} z'^{\mathbf{k}}
 \end{equation}
where $z'^{\mathbf{k}}= z_2^{k_2} z_3^{k_3} \cdots z_{g}^{k_{g}}$ for varying $n\in\Z$ and varying multi-index  $\mathbf{k}=(k_{2},  \cdots ,  k_{g})$ of positive integers,
 constitutes an orthogonal basis of $\Fnug(\VgC)$ with
 \begin{align*}
\normnu{e^{\alpha,{{\nuw}} }_{n,\mathbf{k}}}^2 =  \sqrt{\pi} 2^{|\mathbf{k}|} \mathbf{k}!  \left(\dfrac{1}{2\nuw } \right)^{\frac{2g-1}2 + |\mathbf{k}|} e^{\frac{2\pi^2}{{{\nuw}} }(n+\alpha)^2},
\end{align*}
where $|\mathbf{k}|= k_2+k_3 + \cdots + k_{g}$, and $\mathbf{k}!= k_2!k_3! \cdots k_{g}!$. The corresponding
reproducing kernel function is given explicitly in terms of the theta function (given through \ref{RiemannTheta}) by
\begin{equation}\label{Kernel-explicit}
 K(u,v) =  \left(\frac 1{2\pi\nuw}\right)^{1/2} \left({2\nuw}\right)^{-g} e^{\frac{\nuw}2 (z^2 + \overline{w}^2)}\theta_{\alpha,0} \left(z-\overline{w} \bigg |  {2i\pi}\right)  e^{\widetilde{H}(\widetilde{u},\widetilde{v})}.
   \end{equation}
   for $u= z\omega+\widetilde{u}$ and $v= w\omega+\widetilde{v}$, so that
 \begin{equation*} 
 f(u) =  \int_{v\in \Lambda(\Gamma_1)} K(u,v) f(v) e^{- H(v,v)} \mes(v)
 \end{equation*}
 for every $f\in \Fnug(\VgC)$.

The layout of the paper is as the following.
In Section 2, we give some basic properties of the space $\mathcal{F}^{2}_{\Z \omega, \chi}(\VgC)$ that we need in the elaboration of our main results. We begin with the non-triviality condition and next we describe their elements by giving their explicit expansions.
Section 3, is devoted to the exact statement of the main results and to their proofs. Mainely, we construct an explicit orthonormal basis of
$\mathcal{F}^{2}_{\Z \omega, \chi}(\VgC)$ and express its reproducing kernel function in terms of the modified Riemann theta function.

\section{Preliminaries and first properties}
Let  $\Gamma =\Z \omega $ be a  discrete subgroup of rank one in $(\VgC, +)$ for given nonzero vector in $\omega\in \VgC \setminus\{0\}$.
 The space $\VgC$ can be seen as direct sum of $\VuC:=\C \omega$ and its orthogonal $\VuoC$ with respect to $H$, so that one can split the hermitian form $H$ as
 \begin{equation}\label{splitH}
H(u,v)= z\overline{w} H(\omega, \omega) + \widetilde{H}(\widetilde{u}, \widetilde{v})
\end{equation}
for $u= z\omega + \widetilde{u}$ and $v= w\omega + \widetilde{v}$, where $\widetilde{u}$ is the orthogonal projection of $u$ in $\VuoC$ and $\widetilde{H}$ is the restriction of $H$ to $\VuoC$. The orthogonal $\VuoC$ is generated by some $\C$-linearly independent vectors $\omega_{2},  \cdots,  \omega_{g} \in \VgC$, so that any element $u$ in $\VgC$ can be identified its coordinates in the basis $\{\omega, \omega_2, \cdots , \omega_{g}\}$, i.e.
 $$
 u= z \omega + z_2 \omega_2 \cdots +z_{g} \omega_{g} =  (z ; z_2 \cdots,z_{g})=( z; z').
 $$
  Moreover, the fundamental cell can be identified to
\begin{equation}\label{FundCell}
 \Lambda(\Gamma)=([0,1]\times \R) \times \VuoC.
\end{equation}
   For simplicity, we have to choose $\omega_j$ such that
    \begin{equation}\label{orth-basis}
H(\omega_i,\omega_j)= \delta_{ij} H(\omega, \omega); \qquad i,j=1,2, \cdots.
\end{equation}
The functional equation \eqref{IntrFctEq1} can be written in the
complex coordinates $(z, z')$, with respect to such fixed basis, as follows:
\begin{equation}\label{func-eqm1}
f(z+m, z') = \widetilde{\chi}(m) e^{{{\nuw}}  (z+\frac{m}2) m}f(z, z'),
\end{equation}
for every  $m\in \Z$,  where we have set
 \begin{equation}\label{nu}
{{\nuw}} = H(\omega, \omega)>0 \quad \mbox{ and } \quad \widetilde{\chi}(m)=\chi(m\omega).
\end{equation}

 \begin{remark}\label{remark1}
 In the functional equation \eqref{func-eqm1},  the $z'$ can be seen as parameter. Thus we are dealing only with the one complex variable $z$. Nevertheless, the condition of $L^2$-integrable will involves such parameter as we will see later.
 \end{remark}

 Below, we provide sufficient and necessary condition on the mapping $\chi$ to
the functional space $\Fnug (\VgC)$ be nonzero.

\begin{lemma}\label{lemSC-trivial}
If the space $\Fnug(\VgC)$ is non-trivial, then $\chi$  is a character.
\end{lemma}

\begin{proof} This is an immediate consequence of the fact $f(u+(\gamma+\gamma'))=f((u+\gamma)+\gamma')$, for every  $u\in\VgC$ and $\gamma,\gamma'\in\Z\omega $. Indeed, let $f$ be a nonzero function in $\Fnug(\VgC)$ and $u_0 \in \VgC$ such that $f(u_0)\ne 0$. Direct computation yields
\begin{align}\label{Eq1}
f(u_0+(\gamma+\gamma'))  
                           =\chi(\gamma+\gamma')e^{ H\left(u_0+\frac{\gamma}2, \gamma\right) }e^{ H\left(u_0+\frac{\gamma}2,\gamma'\right) }e^{ H\left(\frac{\gamma'}2, \gamma+\gamma'\right) }f(u_0)
\end{align}
as well as
\begin{align}
f((u_0+\gamma)+\gamma')  &= \chi(\gamma')  e^{ H\left(u_0+\gamma+\frac{\gamma'}2, \gamma'\right) } f(u_0+\gamma)  \nonumber \\
&=\chi(\gamma)\chi(\gamma')e^{ H\left(\frac{\gamma+\gamma'}2, \gamma'\right) }  e^{ H\left(u_0+\frac{\gamma}2, \gamma'\right) }  e^{ H\left(u_0+\frac{\gamma}2, \gamma\right) } f(u_0).\label{Eq2}
\end{align}
For instance, by equating the right hand sides of \eqref{Eq1} and \eqref{Eq2}, we see that
 $\chi$ satisfies the Riemann-Dirac quantization condition
\begin{align}\label{character}
 \chi(\gamma+\gamma') = e^{i \Im(H(\gamma, \gamma'))} \chi(\gamma) \chi(\gamma')
 \end{align}
for every $\gamma, \gamma'\in \Z \omega$. Now, since $H$ takes real values on $\Z \omega\times \Z \omega$, we get
$\Im(H(\gamma, \gamma'))=0$ and therefore the cocycle condition \eqref{character} becomes equivalent to say that $\chi$ is a character of $\Z \omega$.
 \end{proof}

\begin{remark}
The Riemann-Dirac quantization condition RDQ given through \eqref{character} can be interpreted geometrically and algebraically as follows:
\begin{itemize}
  \item Geometric interpretation: The RDQ condition is equivalent to that the complex valued
  function $J_{,\chi}(\gamma,u):=\chi(\gamma )e^{i \Im(H(u,\gamma))}$, on $\Gamma\times\VgC$, be an
   automorphy factor satisfying the chain rule,
   $J_{,\chi}(\gamma_1+\gamma_2,u)=J_{\chi}(\gamma_1,u+\gamma_2)J_{\chi}(\gamma_2,u)$, so that the mapping
   $\phi_\gamma(u;v) := (u+\gamma; \chi(\gamma )e^{i \Im(H(u,\gamma))}.   v)$ defines an action of $\Gamma$ on $\VgC\times\C$ and the
   associated quotient space $(\VgC\times\C)/\Gamma$ is a line
   bundle over $\VgC/\Gamma$ with fiber $\C=\tau^{-1}([u])$, with $\tau$ being the canonical projection map  $\tau: (\VgC\times\C)/\Gamma\longrightarrow \VgC/\Gamma$.
   Thus, one can regard the space $\Fnug(\VgC)$ as the space of holomorphic sections of the above line bundle
   over the quasi-torus $\VgC/\Gamma$.

  \item Algebraic interpretation: Under the RDQ condition, the map $\chi$ induces a group homomorphism from $(\Gamma,+)$ into the Heisenberg
  group $N_\omega:=(\VgC\times U(1); \cdot_\omega)$ endowed with the law group $\cdot_\omega$ defined by
  $(u;\lambda) \cdot_\omega (\mu;v):= (u+v; \lambda\mu e^{i \Im(H(u,v))})$, by considering the injection map
  $i_\omega: (\Gamma,+) \longrightarrow (\VgC\times U(1), \cdot_\omega)$.
\end{itemize}
\end{remark}

\begin{lemma}\label{lemNC-trivial}
Assume that $\chi$ is a character. Then, the function $\epfo(z,z'):= e^{\frac{{{\nuw}} }2 z^2 + 2i\pi \alpha z}$
belongs to $\Fnug(\VgC)$ and we have
 \begin{equation}\label{norm-epfo}
    \normnu{\epfo}^2 =
\left(\frac{1}{2}\right)^{1/2} \left({\frac{\pi}{\nuw }}\right)^{\frac{2g-1}{2}} e^{2\pi^2 \alpha^2 /\nuw}.
\end{equation}
 \end{lemma}

\begin{proof} Since any character $\chi$ is completely determined by
  $$\chi(m\omega)= e^{2i\pi \alpha m} ; \qquad m \in \Z,$$
for certain fixed real number $\alpha$, the holomorphic function
$\epfo(z,z'):= e^{\frac{{{\nuw}} }2 z^2 + 2i\pi \alpha z}$
satisfies \eqref{func-eqm1},
\begin{equation}\label{func-eqm11}
\epfo(z+m,z') = e^{ 2i\pi \alpha m} e^{{{\nuw}} (z+\frac{m}2) m}\epfo(z,z').
\end{equation}
Moreover, starting from \eqref{IntrScal1} and using \eqref{FundCell} and \eqref{splitH}, we can write
\begin{align*}
     \normnu{\epfo}^2= \int_{([0,1]\times \R) \times \VuoC}
     e^{\frac{\nuw}{2} z^2 + 2i\pi \alpha z} \overline{e^{\frac{{{\nuw}} }2 z^2 + 2i\pi \alpha z}}
     e^{- z\overline{z} H(\omega, \omega)}
     e^{- \widetilde{H}(\widetilde{u}, \widetilde{u})}
     \mes(z) \mes(\widetilde{u}).
 \end{align*}
By Fubini theorem and direct computation, keeping in mind \eqref{orth-basis}, it follows
\begin{align*}
     \normnu{\epfo}^2 &=  \left(\int_{[0,1]} dx\right)
     \left(\int_{\R} e^{- 2\nuw y^2 - 4\pi \alpha y}dy \right)
     \left(\int_{\VuoC} e^{- \widetilde{H}(\widetilde{u}, \widetilde{u})}\mes(\widetilde{u})\right)\\
     &=  \left(\int_{\R} e^{- 2\nuw y^2 - 4\pi \alpha y}dy\right)
         \prod_{j=2}^{g} \left(\int_{\R^2} e^{-\nuw (x^{2}_j+y^{2}_j)}dx_jdy_j\right)\\
   &=  \left(\int_{\R} e^{- 2\nuw y^2 - 4\pi \alpha y}dy\right) \left(\int_{\R} e^{-\nuw t^2} dt\right)^{2(g-1)}
 \end{align*}
Now, making use of the explicit formula for the Gaussian integrals
 \begin{equation} \label{gaussIntegral}
  \int_{\R} e^{-a y^2 + by} dy = \left({\frac{\pi}{a}}\right)^{1/2} e^{b^2/4a}
  \end{equation}
 with special values of $a>0$ and $b\in \C$, keeping in mind \eqref{orth-basis}, we obtain
\begin{align*}
\normnu{\epfo}^2 =  \left(\frac{1}{2}\right)^{1/2} \left({\frac{\pi}{\nuw }}\right)^{\frac{2g-1}{2}} e^{2\pi^2 \alpha^2 /\nuw} .
\end{align*}
 Thus $\normnu{\epfo}^2$ is finite. This completes the proof.
 \end{proof}

\begin{remark}
Lemmas \ref{lemSC-trivial} and \ref{lemNC-trivial} can be reworded as: $\Fnug(\VgC)$ is non trivial if and only if $\chi$ is a character.
\end{remark}

 To establish our main result, we  need  the following.

 \begin{proposition}\label{prop-description}
 The holomorphic function $f$ on $\VgC$ satisfies \eqref{func-eqm1} if and only if
 it can be expanded in series as
   \begin{equation}\label{expansion0}
     f(z, z'):= \sum\limits_{n\in\Z} \psi_{n}(z') e^{\frac{{{\nuw}} }2 z^2 + 2i\pi (\alpha + n) z},
     \end{equation}
     where  the coefficients $\psi_{n} (z')$ are holomorphic functions in $ z' \in \VuoC $.
\end{proposition}

\begin{proof}
Let $f$ be a holomorphic function displaying \eqref{func-eqm1} and let $h$ be a function defined by
\begin{equation}\label{expans}
h(z,z')= e^{ \frac{-\nuw}{2} z^2-2i\pi \alpha z} f(z,z'),
\end{equation}
we get
\begin{align*}
h(z+m,z')&=e^{ \frac{-\nuw}{2} (z+m)^2-2i\pi \alpha (z+m)} f(z+m,z') \\
         &=e^{ \frac{-\nuw}{2} z^2-2i\pi \alpha z} f(z,z')= h(z,z'),
\end{align*}
for every $(z,z')\in\C\times\C^{g-1}$, and for every $m\in\Z$ .
 Then, for every fixed $z'$, the  function $$ h_{z'}(z):=h(z,z') ,\quad\quad z\in\C$$
 is a $\Z$-simply periodic function with respect to $z$ and in the $x$-direction,  with $z=x+iy$; $x,y\in\R$. Therefore, it can be expanded as (\cite{Bargmann1961,Jones-Singerman1987}):
\begin{equation}\label{func}
    h_{z'}(z) =  \sum\limits_{n\in\Z} \psi_n(z') e^{ 2i\pi n z} .
\end{equation}
The involved series converge absolutely and uniformly on compact subsets of $\C$. The Fourier coefficients $\psi_n(z')$; $n\in\Z$, are given by
\begin{equation}\label{coefficient}
 \psi_n(z') = \int_0^1 h_{z'}(x) e^{-2i\pi n x} dx .
\end{equation}
Based on \ref{coefficient}, the functions $ z' \longmapsto \psi_n(z')$ are holomorphic on the whole $\VuoC$ for the integrand function.
 Finally, using \eqref{expans} and \eqref{func}, we get that
     \begin{align}\label{expansion01}
     f(z, z') =\sum\limits_{n\in\Z} \psi_n(z') e^{\frac{{{\nuw}} }2 z^2 + 2i\pi (\alpha + n) z}.
     \end{align}
\end{proof}

Furthermore, we have the following

\begin{proposition}\label{orthset}
The set of functions
 \begin{equation}\label{def-ena}
 \epf(z,z') :=  \epfon (z,z') = e^{\frac{\nuw}2 z^2 + 2i\pi (\alpha + n) z}; \quad n\in\Z,
 \end{equation}
 constitutes an orthogonal system in $\Fnug(\VgC)$ with
    \begin{equation}\label{norm-def3}
    \normnu{\epf}^2 = \left(\dfrac{1}{2}\right)^{1/2} \left(\dfrac{\pi}{\nuw}\right)^{\frac{2g-1}2} e^{\frac{2\pi^2}{{{\nuw}} }(n+\alpha)^2}.
    \end{equation}
\end{proposition}

\begin{proof}
 In view of Lemma \ref{lemNC-trivial}, we see that $\epf(z,z') :=  \epfon (z,z') \in \Fnug(\VgC)$ and $\normnu{\epf}^2= \normnu{\epfon}^2$. What is to prove is the orthogonality, indeed by proceeding as in the proof of Lemma \ref{lemNC-trivial}, we get
 \begin{align}
\scalnue{\epf, \eqf}&= \int_{([0,1]\times \R) \times \VuoC} \epf(z,z') \overline{\eqf(z,z')}
                        e^{- z\overline{z}H(\omega,\omega)} e^{- \widetilde{H}(\widetilde{u}, \widetilde{u})} \mes(z) \mes(\widetilde{u})  \nonumber \\
                  &= \left(\int_0^1e^{ 2i\pi(n-n') x} dx\right) \left(\int_{\R} e^{-2{{\nuw}}  y^2 - 2\pi (2\alpha +n+n')y} dy\right)
                   \left(\int_{\R} e^{-\nuw  t^2} dt\right)^{2(g-1)} \nonumber\\
                   &= \delta_{n,n'}  \normnu{\epfon}^2. \label{scalarnorm}
\end{align}
Whenever, $\{\epf; \, n\in\Z\}$ is an orthogonal set in $\Fnug(\VgC)$. This completes the proof.
\end{proof}

 \section{Main results}
The first main result of this paper is Theorem 3.1 below which gives concrete description of the space $\Fnug(\VgC)$. Precisely, we assert

\begin{theorem}\label{ThmExpansion}
A function $f$  belongs to $\Fnug(\VgC)$ if and only if it can be written as
\begin{equation}\label{expansion}
     f(z,z'):= \sum\limits_{n\in\Z} \psi_n(z') e^{\frac{{{\nuw}} }2 z^2 + 2i\pi (\alpha + n) z},
     \end{equation}
  where the series converge uniformly on each thinner strip. The functions  $\psi_n$ are holomorphic on $\VuoC$ and satisfying the growth condition
      \begin{align}\label{gcond}
       \left(\dfrac{\pi}{2{{\nuw}} }\right)^{1/2}\sum\limits_{n\in\Z}e^{\frac{2\pi^2}{\nuw}(n+\alpha)^2}
       \|\psi_n\|_{L^2(\VuoC, \widetilde{H})}^{2}  <+\infty ,
       \end{align}
      where $\|\psi_n\|_{L^2(\VuoC, \widetilde{H})}^{2}$ denotes
      the square norm of $\psi_n$ in the Hilbert space $L^2(\VuoC, e^{- \widetilde{H}(\widetilde{u},\widetilde{u})}\mes(\widetilde{u}))$.
\end{theorem}

\begin{proof}
We make use of Proposition \ref{prop-description}, to expand any holomorphic function $f$ satisfying \eqref{func-eqm11} as
       \begin{align*}
        f(z,z') 
        &= e^{\frac{\nuw}2 z^2 + 2i\pi\alpha  z} \sum\limits_{n\in\Z} e^{ -2\pi n y}\psi_n(z') e^{ 2i\pi n x}\\
        &= e^{\frac{\nuw}2 z^2 + 2i\pi\alpha  z} \psi_{(y;z')}(x) ,
       \end{align*}
where $\psi_n$ are holomorphic functions on $\VuoC$, and where we have set
       \begin{align}\label{perio-fct}
        \psi_{(y;z')}(x) = \sum\limits_{n\in\Z} e^{ -2\pi n y}\psi_n(z') e^{ 2i\pi n x}.
       \end{align}
  So that
\begin{align*}
\normnu{f}^2 &=\int_{[0,1]\times\R\times \VuoC} |f(u)|^2 e^{- H(u,u)} \mes(u) \\
             &=\int_{[0,1]\times\R\times \VuoC} e^{-2\nuw y^2 -4\pi \alpha y}  \left|\psi_{(y;z')}(x) \right|^2
             e^{- \widetilde{H}(\widetilde{u}, \widetilde{u})} dx dy \mes(\widetilde{u}).
\end{align*}
Thus, by Fubini theorem we obtain
\begin{align}\label{Fubinin}
\normnu{f}^2  &= \int_{\R\times \VuoC}   e^{-2\nuw y^2 -4\pi \alpha y}
    \left(  \int_{[0,1]} \left|\psi_{(y;z')}(x)\right|^2 dx\right)
    e^{- \widetilde{H}(\widetilde{u}, \widetilde{u})} dy\mes(\widetilde{u}).
\end{align}
Whence, if $\normnu{f}^2$ is finite, then $\int_{[0,1]} \left|\psi_{(y;z')}(x)\right|^2 dx$ is finite
for every fixed $y\in \R$ and $ z'\in \VuoC$. Therefore, the periodic function
$x \longrightarrow \psi_{(y;z')}(x)$
 belongs to the Hilbert space $L^2([0,1];dx)$, where the series converges uniformly on $[0,1]$.
By applying the Parseval identity, keeping in mind \eqref{perio-fct}, we obtain
\begin{align*}
\int_{[0,1]} \left|\psi_{(y;z')}(x)\right|^2 dx =  \sum\limits_{n\in\Z} e^{ -4\pi n y}\left| \psi_n(z') \right|^2.
\end{align*}
Inserting this in \eqref{Fubinin} yields
\begin{align*}
\normnu{f}^2  &=  \sum\limits_{n\in\Z}  \left(\int_{\R}  e^{-2\nuw y^2 -4\pi (\alpha+n) y} dy\right)
\left(\int_{\VuoC} \left| \psi_n(z') \right|^2 e^{- \widetilde{H}(\widetilde{u}, \widetilde{u})} \mes(\widetilde{u})\right).
\end{align*}
The value of the first integral in the right hand side of the last equality follows from \eqref{gaussIntegral}. The second one leads to the square norm of $\psi_n$ in the Hilbert space
$L^2(\VuoC, e^{- \widetilde{H}(\widetilde{u},\widetilde{u})}\mes(\widetilde{u}))$.
Thence,
\begin{align*}
\normnu{f}^2  &=  \left({\frac{\pi}{2\nuw }}\right)^{1/2} \sum\limits_{n\in\Z}   e^{\frac{2\pi^2}{\nuw} (\alpha+n)^2 }
\|\psi_n\|_{L^2(\VuoC, \widetilde{H})}^{2} .
\end{align*}
This completes the proof.
\end{proof}

\begin{remark} \label{Bargm}
In view of  the growth condition \eqref{gcond}, it is clear that the holomorphic function $z'\longmapsto \psi_n(z'),n \in \Z,$ belongs to the usual Bargmann space
$$\mathcal{B}^{2,{{\nuw}} }(\VuoC)=\left\{\mbox{  h holomorphic on \,} \VuoC; \, \int_{\VuoC} |h(\widetilde{u})|^2
e^{-\widetilde{H}(\widetilde{u}, \widetilde{u})} \mes(\widetilde{u})  <+\infty \right\} $$
and therefore can be expanded as
$$\psi_n(z')=\sum\limits_{\mathbf{k}\in (\Z^+)^{g-1}} a_{n,\mathbf{k}} z'^{\mathbf{k}} ; \qquad a_{n,\mathbf{k}}\in\C $$
uniformly convergent on compact sets of $\VuoC$. Here $z'^{\mathbf{k}}=z_2^{k_2}z_3^{k_3}\cdots z_{g}^{k_{g}}$ for $\mathbf{k}=(k_2, \cdots,k_{g}) \in(\Z^+)^{g-1}$.
Thence, it follows from Theorem \eqref{ThmExpansion} that any function $f$  belonging to $\Fnug(\VgC)$ has the following expansion series
 \begin{equation}\label{expansionz}
  f(z,z'):= \sum\limits_{n\in\Z} \sum\limits_{\mathbf{k}\in (\Z^+)^{g-1}}
  a_{n,\mathbf{k}}  e^{\frac{{{\nuw}} }2 z^2 + 2i\pi (\alpha + n) z} z'^{\mathbf{k}}.
     \end{equation}
\end{remark}
For every $f\in\Fnug(\VgC) $, and every $(z,z')\in\C\times\C^{g-1}$, we have
$$  |f(z,z') |\leq\|f\|_{\Z\omega,{{\nuw}} }\left(\dfrac{2\nuw} {\pi}\right)^{1/4} e^{\frac{\nuw}{4}  (z^2+\overline{z}^2)}\left( \sum\limits_{n\in\Z} e^{\frac{-2\pi^2}{{{\nuw}} }(n+\alpha)^2+ 2i\pi (\alpha + n)( z-\overline{z})}\frac{|\psi_n(z')|}{\|\psi_n\|_{L^2(\VuoC, \widetilde{H})}^{2}}\right)^{\frac{1}{2}}. $$
It should be understood that the summation is over $n\in \Z$ such that $\psi_n$ is not zero.

This follows by applying appropriately the Cauchy Schwartz inegality to the expansion series of $f$ given throug \eqref{expansion01}. So that for every compact set $K$ of $\C\times\C^{g-1}$, there is a constant  $C_{K}\geq 0$ such that
$$ |f(z,z') | \leq C_{K}\|f\|_{\Z\omega,{{\nuw}} } \quad\quad; (z,z')\in K .$$
This can be used in a classical way to prove the $\Fnug(\VgC)$ is a Hilbert space .

Moreover, we assert

\begin{theorem}  \label{ThmbasisO}
The $(\Z\omega, \chi)$-theta Fock-Bargmann space $\Fnug(\VgC)$ is a reproducing kernel Hilbert space and
the set of functions
 \begin{equation}\label{basis}
 e^{\alpha,{\nuw} }_{n,\mathbf{k}}(z,z')=: 
 e^{\frac{\nuw}2 z^2 + 2i\pi (\alpha + n) z} z'^{\mathbf{k}}
 \end{equation}
 for varying $n\in\Z$ and varying multi-index  $\mathbf{k}=(k_{2},  \cdots ,  k_{g})$ of positive integers,
 constitutes an orthogonal basis of $\Fnug(\VgC)$ with
 \begin{align}\label{norm-basiss}
\normnu{e^{\alpha,{{\nuw}} }_{n,\mathbf{k}}}^2 =  \sqrt{2\pi\nuw}  \left(\dfrac{1}{2\nuw } \right)^{g}
\left(\dfrac{1}{\nuw } \right)^{|\mathbf{k}|}\mathbf{k}! e^{\frac{2\pi^2}{{{\nuw}} }(n+\alpha)^2},
\end{align}
where $|\mathbf{k}|= k_2+k_3 + \cdots + k_{g}$, and $\mathbf{k}!= k_2!k_3! \cdots k_{g}!$.
\end{theorem}

\begin{proof}
Note first that the functions $e^{\alpha,{\nuw} }_{n,\mathbf{k}}$ belong to $\Fnug(\VgC)$ according to Proposition \ref{orthset}.
Orthogonality of $e^{\alpha,{\nuw} }_{n,\mathbf{k}}$ follows from the orthogonality of $\epf$. More precisely,
doing so as in Proposition \ref{orthset}, we get
\begin{align}
\scalnue{e^{\alpha,{{\nuw}} }_{n,\mathbf{k}},e^{\alpha,{{\nuw}} }_{n',\mathbf{k}'}}
& =  \delta_{n,n'} \left(\dfrac{\pi}{2{{\nuw}} }\right)^{1/2}e^{\frac{2\pi^2}{{{\nuw}} }(n+\alpha)^2}  \left(\int_{\C^{g-1}} {z}^{\mathbf{k}} \overline{z}^{\mathbf{k'}}  e^{- \widetilde{H}(\widetilde{u}, \widetilde{u})}  \mes(\widetilde{u})\right).  \nonumber
\end{align}
Now, using the well established fact
\begin{align*}
\int_{\C^{g-1}} {z}^{\mathbf{k}} \overline{z}^{\mathbf{k'}}  e^{- \widetilde{H}(\widetilde{u}, \widetilde{u})}  \mes(\widetilde{u})=
\left(\dfrac{1}{2\nuw } \right)^{g-1} \left(\dfrac{1}{\nuw } \right)^{|\mathbf{k}|} \mathbf{k}! \delta_{\mathbf{k},\mathbf{k}'} 
\end{align*}
it follows
\begin{align*}
\scalnue{e^{\alpha,{{\nuw}} }_{n,\mathbf{k}},e^{\alpha,{{\nuw}} }_{n',\mathbf{k}'}}
=  \left(\dfrac{\pi}{2{{\nuw}} }\right)^{1/2} e^{\frac{2\pi^2}{{{\nuw}} }(n+\alpha)^2}
\left(\dfrac{1}{2\nuw } \right)^{g-1} \left(\dfrac{1}{\nuw } \right)^{|\mathbf{k}|} \mathbf{k}! \delta_{n,n'} \delta_{\mathbf{k},\mathbf{k}'} .
\end{align*}
Then, we get
\begin{align}
\normnu{e^{\alpha,{{\nuw}} }_{n,\mathbf{k}}}^2 =  \sqrt{2\pi\nuw}  \left(\dfrac{1}{2\nuw } \right)^{g}
\left(\dfrac{1}{\nuw } \right)^{|\mathbf{k}|}\mathbf{k}! e^{\frac{2\pi^2}{{{\nuw}} }(n+\alpha)^2},
\end{align}
Let $f\in\Fnug(\VgC)$ ,
in view of Theorem \ref{ThmExpansion} and Remark \eqref{Bargm}, we have
 \begin{equation}
  f(z,z'):= \sum\limits_{n\in\Z} \sum\limits_{\mathbf{k}\in (\Z^+)^{g-1}}
  a_{n,\mathbf{k}}  e^{\alpha,{{\nuw}} }_{n,\mathbf{k}}(z,z')
\end{equation}
 Furthermore, the series converge uniformly on compact sets of $\VgC$.  So that
 \begin{align*}
\scalnue{f,e^{\alpha,{{\nuw}} }_{n',\mathbf{k'}}}
=\sum\limits_{n\in\Z; \mathbf{k} \in(\Z^+)^{g-1}}a_{n,\mathbf{k}}\scalnue{e^{\alpha,{{\nuw}} }_{n,\mathbf{k}},e^{\alpha,{{\nuw}} }_{n',\mathbf{k'}}} =a_{n',\mathbf{k'}} \normnu{e^{\alpha,{{\nuw}} }_{n',\mathbf{k'}}}^2
 \end{align*}
 by orthogonality of $e^{\alpha,{{\nuw}} }_{n,\mathbf{k}}$.
 Thus, for $f$ belonging to the orthogonal of $Span\{e^{\alpha,{{\nuw}} }_{n,\mathbf{k}} ,n\in\Z \}$ in the Hilbert space  $\Fnug(\VgC)$, we have
 $\scalnue{f,e^{\alpha,{{\nuw}} }_{n',\mathbf{k'}}}=0$ and therefore $a_{n',\mathbf{k'}}=0$ for every
 $n'\in\Z$ and $\mathbf{k'} \in(\Z^+)^{g-1}$. This shows that $f$ is identically zero and so
 the orthogonal of $Span\{e^{\alpha,{{\nuw}} }_{n,\mathbf{k}} ,n\in\Z \}$ reduces to zero. This completes the proof.
\end{proof}
As immediate consequence of provious theorem, we have
\begin{corollary}
The reproducing kernel $K(u,v)$ of $\Fnug(\VgC)$ is given by in terms of the theta function $\theta_{\alpha,0}$ in \eqref{RiemannTheta} as
\begin{equation}\label{Kernel-explicit}
 K(u,v) =  \left(\frac 1{2\pi\nuw}\right)^{1/2} \left({2\nuw}\right)^{-g} e^{\frac{\nuw}2 (z^2 + \overline{w}^2)}\theta_{\alpha,0} \left(z-\overline{w} \bigg |  {2i\pi}\right) e^{ \widetilde{H}(\widetilde{u}, \widetilde{v})},
 \end{equation}
   where  $u= z\omega+\widetilde{u}$ and $v= w\omega+\widetilde{v}.$
\end{corollary}

\begin{proof} Based on Theorem \eqref{ThmbasisO}, the reproducing kernel function of $\Fnug(\VgC)$ is given by
\begin{align*}
K(u,v) &=\sum\limits_{n\in\Z}\sum\limits_{\mathbf{k} \in (\Z^+)^{g-1}}
\frac{ e^{\alpha,\nuw}_{n,\mathbf{k}}(u)   \overline{e^{\alpha,\nuw}_{n,\mathbf{k}}(v)} }{ \normnu{e^{\alpha,{{\nuw}} }_{n,\mathbf{k}}}^2 }\\
 &=\sum\limits_{n\in\Z}\sum\limits_{\mathbf{k} \in (\Z^+)^{g-1}}
\frac{ e^{\frac{\nuw}2 z^2 + 2i\pi (\alpha + n) z}z'^{\mathbf{k}}   \overline{e^{\frac{\nuw}2 w^2 + 2i\pi (\alpha + n) w} w'^{\mathbf{k}}} }
{ \sqrt{2\pi\nuw} \left(\dfrac{1}{2\nuw}\right)^{g}\left(\dfrac{1}{\nuw} \right)^{|\mathbf{k}|}\mathbf{k}! e^{\frac{2\pi^2}{\nuw}(n+\alpha)^2} }\\
&= \left(\frac 1{2\pi\nuw}\right)^{1/2} \left({2\nuw}\right)^{-g} e^{\frac{\nuw}2 (z^2 + \overline{w}^2)}\theta_{\alpha,0} \left(z-\overline{w} \bigg |  {2i\pi}\right)
   e^{ \widetilde{H}(\widetilde{u}, \widetilde{v})},
\end{align*}
where  $u= z\omega+\widetilde{u}$ and $v= w\omega+\widetilde{v}.$
\end{proof}

\end{document}